\documentclass[a4paper,10pt]{amsart}

\usepackage[utf8]{inputenc} % Linux
\usepackage{amsmath}
\usepackage{amsbsy}
\usepackage{amssymb}
\usepackage{amscd,amsthm}
\usepackage{verbatim}
\usepackage[english]{babel}
\usepackage{dsfont}
\usepackage{anysize}
\usepackage{hyperref}

\newcommand{\tl}{\text{li}}

\newcommand{\R}{\mathds{R}}

\newcommand{\p}{\phantom}
\newcommand{\q}{\quad}

\newtheorem{thm}{Theorem}[section]

\newtheorem{kor}[thm]{Corollary}

\newtheorem{prop}[thm]{Proposition}

\theoremstyle{definition}

\theoremstyle{remark}
\newtheorem*{rema}{Remark}

\title[Estimates for $\pi(x)$ for large values of $x$ and Ramanujan's prime counting inequality]{Estimates for $\pi(x)$ for large values of $x$ and \\ Ramanujan's prime counting inequality}
\author{Christian Axler}
\address{Institute of Mathematics\\ Heinrich-Heine University Düsseldorf\\
40225 Düsseldorf, Germany}
\email{christian.axler@hhu.de}
\subjclass[2010]{Primary 11N05; Secondary 11A41}
\keywords{Chebyshev's $\vartheta$-function, prime counting function, Ramanujan's prime counting inequality}

\begin{document}
\maketitle
\vspace*{-5mm}
\begin{center}
{\small March 2017}
\end{center}
\vspace*{-2mm}
\begin{abstract}
In this paper we use refined approximations for Chebyshev's $\vartheta$-function to establish new explicit estimates for the prime counting function $\pi(x)$, which improve the current best estimates for large values of $x$. As an application we find an upper bound for the number $H_0$ which is defined to be the smallest positive integer so that Ramanujan's prime counting inequality holds for every $x \geq H_0$.
\end{abstract}

\section{Introduction}
Let $\pi(x)$ denotes the number of primes not exceeding $x$. Since there are infinitely many primes, we have $\pi(x) \to \infty$ for $x \to \infty$. In 1793, Gau{\ss} \cite{gauss} stated a conjecture concerning an asymptotic behavior of $\pi(x)$, namely
\begin{equation}
\pi(x) \sim \text{li}(x) \quad\quad (x \to \infty), \tag{1.1} \label{1.1}
\end{equation}
where the \emph{logarithmic integral} $\text{li}(x)$ defined for every real $x \geq 0$ as
\begin{equation}
\text{li}(x) = \int_0^x \frac{dt}{\log t} = \lim_{\varepsilon \to 0} \left \{ \int_{0}^{1-\varepsilon}{\frac{dt}{\log t}} + \int_{1+\varepsilon}^{x}{\frac{dt}{\log t}} \right \} = \int_2^x \frac{dt}{\log t} + 1.04516\ldots. \tag{1.2} \label{1.2}
\end{equation}
The asymptotic formula \eqref{1.1} was proved independently by Hadamard \cite{hadamard1896} and by de la Vall\'{e}e-Poussin \cite{vallee1896} in 1896, and is known as the \textit{Prime Number Theorem}. In his later paper \cite{vallee1899}, where he proved the existence of a zero-free region for the Riemann zeta-function $\zeta(s)$ to the left of the line $\text{Re}(s) = 1$, de la Vall\'{e}e-Poussin also estimated the error term in the Prime Number Theorem by showing
\begin{equation}
\pi(x) = \text{li}(x) + O(x \exp(-a\sqrt{\log x})), \tag{1.3} \label{1.3}
\end{equation} 
where $a$ is a positive absolute constant. The work of  Korobov \cite{korobov1958} and Vinogradov \cite{vinogradov1958} implies a much better result, namely that there is a positive absolute constant $c$ so that
\begin{displaymath}
\pi(x) = \text{li}(x) + O \left( x \exp \left( - c (\log x)^{3/5} (\log \log x)^{-1/5} \right) \right).
\end{displaymath}
In 1901, von Koch \cite{koch1901} deduced under the assumption that the Riemann hypothesis is true a remarkable refinement of the error term, namely
%of the difference $|\pi(x) - \text{li}(x)|$, which is given by
\begin{equation}
\pi(x) = \text{li}(x) + O(\sqrt{x} \log x). \tag{1.4} \label{1.4}
\end{equation}
In 2000, Panaitopol \cite[p. 55]{pan3} gave another asymptotic formula for the prime counting function by showing that for each positive integer $m$, we have
\begin{equation}
\pi(x) = \frac{x}{ \log x - 1 - \frac{k_1}{\log x} - \frac{k_2}{\log^2x} - \ldots -  \frac{k_m}{\log^m x}} + O \left( \frac{x}{\log^{m+2}x}  \right), \tag{1.5} \label{1.5}
\end{equation}
where the positive integers $k_1, \ldots, k_m$ are defined by the recurrence formula
\begin{displaymath}
k_m + 1!k_{m-1} + 2!k_{m-2} + \ldots + (m-1)!k_1 = m \cdot m!.
\end{displaymath}
For instance, we have
\begin{displaymath}
k_1 = 1, \, k_2 = 3, \, k_3 = 13, \, k_4 = 71, \, k_5 = 461, \, k_6 = 3\,441.
\end{displaymath}
Hence, the asymptotic formula \eqref{1.5} implies that the inequality that
\begin{equation}
\pi(x) > \frac{x}{\log x - 1 - \frac{1}{\log x} - \frac{3}{\log^2 x} - \ldots - \frac{k_n}{\log^n x}} \tag{1.6} \label{1.6}
\end{equation}
holds for every positive integer $n$ and all sufficently large values of $x$. The first result in this direction is from 1962 and is due to Rosser and Schoenfeld \cite[Corollary 1]{rosser1962}. They showed that the inequality
\begin{equation}
\pi(x) > \frac{x}{\log x} \tag{1.7} \label{1.7}
\end{equation}
holds for every $x \geq 17$. In 1998, Dusart \cite[Th\'{e}or\`{e}me 1.10]{pd1} obtained that
\begin{displaymath}
\pi(x) > \frac{x}{\log x - 1}
\end{displaymath}
for every $x \geq 5393$. The current best result concerning an upper bound which corresponds to
the first terms of \eqref{1.5} is given in \cite[Korollar 1.24]{axler2013} and states that
\begin{displaymath}
\pi(x) > \frac{x}{\log x - 1 - \frac{1}{\log x}}
\end{displaymath}
for every $x \geq 468\,049$. In the following theorem, we make a first progress in finding the smallest positive integer $N_0$ so that the inequality \eqref{6.2} holds for $n=2$ and every $x \geq N_0$.

\begin{thm} \label{thm101}
The inequality
\begin{displaymath}
\pi(x) > \frac{x}{\log x - 1 - \frac{1}{\log x} - \frac{3}{\log^2 x}}
\end{displaymath}
holds for every $x$ such that $65\,405\,887 \leq x \leq 2.73 \cdot 10^{40}$ and every $x \geq e^{580044/13}$.
\end{thm}

Integration of parts in \eqref{1.3} implies that the asymptotic expansion
\begin{equation}
\pi(x) = \frac{x}{\log x} + \frac{x}{\log^2 x} + \frac{2x}{\log^3 x} + \frac{6x}{\log^4 x} + \ldots + \frac{(m-1)! x}{\log^mx}+ O \left( \frac{x}{\log^{m+1} x} \right) \tag{1.8} \label{1.8}
\end{equation}
holds for each positive integer $m$, which implies that there exists a smallest positive integer $g_1(n) \geq 2$ so that the inequality
\begin{displaymath}
\pi(x) > \frac{x}{\log x} + \frac{x}{\log^2 x} + \frac{2x}{\log^3 x} + \frac{6x}{\log^4 x} + \frac{24x}{\log^5 x} + \ldots + \frac{(n-1)!x}{\log^nx}
\end{displaymath}
holds for every positive integer $n$ and every $x \geq g_1(n)$. Again, the inequality \eqref{1.7}, obtained by Rosser and Schoenfeld \cite[Corollary 1]{rosser1962}, was the first result concerning an upper bound which corresponds to the first terms of \eqref{1.8}. Dusart \cite[Th\'{e}or\`{e}me 1.10]{pd1} found in 1998 that $g_1(2) = 599$. In 2010, he \cite[Theorem 6.9]{dusart2010} improved his own result by showing that $g_1(3) = 88\,783$. In the following theorem, we go one step further by finding an upper bound for the smallest positive integer $g_1(4)$.
% so that the inequality
%\begin{equation}
%\pi(x) > \frac{x}{\log x} + \frac{x}{\log^2 x} + \frac{2x}{\log^3 x} + \frac{6x}{\log^4 x} \tag{1.9} \label{1.9}
%\end{equation}
%holds for every $x \geq x_1(4)$.

\begin{thm} \label{thm102}
The inequality
\begin{displaymath}
\pi(x) > \frac{x}{\log x} + \frac{x}{\log^2 x} + \frac{2x}{\log^3 x} + \frac{6x}{\log^4 x}
\end{displaymath}
holds for every $x$ such that $10\,384\,261 \leq x \leq 2.73 \cdot 10^{40}$ and every $x \geq e^{6719}$.
\end{thm}

As an application of the estimates for the prime counting function which hold for all sufficiently large values of $x$, we consider an inequality established by Ramanujan. In one of his notebooks (see Berndt \cite{berndt1994}), Ramanujan used \eqref{1.8} with $n = 5$ to find that
\begin{displaymath}
\pi(x)^2 - \frac{ex}{\log x} \pi \left( \frac{x}{e} \right)  = - \frac{x^2}{\log^6x} + O \left( \frac{x}{\log^7 x} \right)
\end{displaymath}
and concluded that the inequality
\begin{equation}
\pi(x)^2 < \frac{ex}{\log x} \pi \left( \frac{x}{e} \right) \tag{1.9} \label{1.9}
\end{equation}
holds for all sufficiently large values of $x$. The inequality \eqref{1.9} is called \textit {Ramanujan's prime counting inequality}. The problem arose to find the smallest integer $H_0$ so that the inequality \eqref{1.9} holds for every real $x \geq H_0$. Under the assumption that the Riemann hypothesis is true (RH), Hassani \cite[Theorem 1.2]{hassani2012} has given the upper bound 
\begin{displaymath}
RH \; \Rightarrow \; H_0 \leq 138\,766\,146\,692\,471\,228.
\end{displaymath}
In 2015, Dudek and Platt \cite[Lemma 3.2]{dudekplatt} refined Hassani's result by showing
\begin{equation}
RH \; \Rightarrow \; H_0 \leq 1.15 \cdot 10^{16}. \tag{1.10} \label{1.10}
\end{equation}
Wheeler, Keiper and Galway (see Berndt \cite[p. 113]{berndt1994}) attempted to determine the value of $H_0$, but they failed. Nevertheless, Galway found that the largest prime up to $10^{11}$ for which the inequality \eqref{1.9} fails is $x=38\,358\,837\,677$. Hence
\begin{displaymath}
H_0 > 38\,358\,837\,677.
\end{displaymath}
Dudek and Platt \cite[Theorem 1.3]{dudekplatt} showed by computation that $x = 38\,358\,837\,682$ is the largest integer counterexample below $10^{11}$ and that there are no more failures at integer values before $1.15 \cdot 10^{16}$. Hence the inequality \eqref{1.9} holds unconditionally for every $x \in I_0$, where $I_0 = [38\,358\,837\,683, 1.15 \cdot 10^{16}]$. Together with \eqref{1.10},
\begin{equation}
RH \; \Rightarrow \; H_0 = 38\,358\,837\,683. \tag{1.11} \label{1.11}
\end{equation}
Based on a result of Büthe \cite[Theorem 2]{buethe}, we extend the interval $I_0$, in which the inequality \eqref{1.9} holds unconditionally by showing the following theorem.

\begin{thm} \label{thm103}
Ramanujan's prime counting inequality \eqref{1.9} holds unconditionally for every $x$ such that $38\,358\,837\,683 \leq x \leq 10^{19}$.
\end{thm}

In addition, Dudek and Platt \cite[Theorem 1.2]{dudekplatt} claimed to give an upper bound for $H_0$ which does not depend on the assumption that the Riemann hypothesis is true, namely
\begin{equation}
H_0 \leq e^{9658}. \tag{1.12} \label{1.12}
\end{equation}
%After the present author contacted Platt, one of the authors of \cite{dudekplatt}, to raise some doubts about the correctness of the proof of \eqref{1.6}, Platt confirmed after some discussions that the proof of \eqref{1.6} in its present form is not correct. This incorrectness motivated the present author to write this paper, where we prove the following result. Here, explicit estimates for the prime counting function play an important role.
After the present author raised some doubts about the correctness of the proof of \eqref{1.12}, one of the authors confirmed (email communication) that the proof of \eqref{1.12} given in \cite{dudekplatt} is not correct. This motivated us to write this paper, where we prove the following even stronger result. In our proof, explicit estimates for the prime counting function which hold for all sufficiently large values of $x$ play an important role.

\begin{thm} \label{thm104}
Ramanujan's prime counting inequality \eqref{1.9} holds unconditionally for every real $x \geq e^{9032}$; i.e.
\begin{displaymath}
H_0 \leq e^{9032}.
\end{displaymath}
\end{thm}

In Section 7, we use Theorem \ref{thm103}, Theorem \ref{thm104} and \eqref{1.11} to establish a result concerning a generalized inequality of Ramanujan's prime counting inequality \eqref{1.9}.

\section{On Chebyshev's $\vartheta$-function}

In order to prove Theorem \ref{thm101} and Theorem \ref{thm102}, we first consider Chebyshev's $\vartheta$-function, which is defined by
\begin{displaymath}
\vartheta(x) = \sum_{p \leq x} \log p,
\end{displaymath}
where $p$ runs over primes not exceeding $x$. The prime counting function and Chebyshev's $\vartheta$-function are connected by the well-known identities
\begin{equation}
\pi(x) = \frac{\vartheta(x)}{\log x} + \int_{2}^{x}{\frac{\vartheta(t)}{t \log^{2} t}\ dt}, \tag{2.1} \label{2.1}
\end{equation}
and
\begin{equation}
\vartheta(x) = \pi(x) \log x - \int_{2}^{x}{\frac{\pi(t)}{t}\ dt}, \tag{2.2} \label{2.2}
\end{equation}
which hold for every $x \geq 2$ (see, for instance, Apostol \cite[Theorem 4.3]{ap}). Using \eqref{2.2}, it is easy to see that the Prime Number Theorem is equivalent to
\begin{equation}
\vartheta(x) \sim x \quad\quad (x \to \infty). \tag{2.3} \label{2.3}
\end{equation}
By proving the existence of a zero-free region for the Riemann zeta-function $\zeta(s)$ to the left of the line $\text{Re}(s) = 1$ , de la Vall\'{e}e-Poussin \cite{vallee1899} was abled to bound the error term in \eqref{2.3} by proving
\begin{equation}
\vartheta(x) = x + O(x \exp(-a\sqrt{\log x})), \tag{2.4} \label{2.4}
\end{equation} 
where $a$ is a positive absolute constant. In this direction, we give the following result.

\begin{prop}\label{prop201}
Let $R = 5.573412$. Then,
\begin{equation}
|\vartheta(x) - x| < \frac{\sqrt{8}}{\sqrt{\pi \sqrt{R}}} \, x (\log x)^{1/4} e^{- \sqrt{(\log x)/R}} \tag{2.5} \label{2.5}
\end{equation}
for every $x \geq 3$.
\end{prop}

\begin{proof}
By Mossinghoff and Trudgian \cite[Theorem 1]{mossing}, there are no zeros of the Riemann zeta fuction $\zeta(s)$ for $| \text{Im}(s)| \geq 2$ and
\begin{displaymath}
\text{Re}(s) \geq 1 - \frac{1}{R \log | \text{Im}(s)|}.
\end{displaymath}
Applying this to \cite[Theorem 1.1]{dusart2016}, we get that the required inequality holds for every $x \geq e^{390}$. Further, Trudgian \cite[Theorem 1]{trud} showed that the inequality
\begin{displaymath}
|\vartheta(x) - x| < \frac{\sqrt{8}}{\sqrt{17\pi \sqrt{6.455}}} \, x (\log x)^{1/4} e^{- \sqrt{(\log x)/6.455}}
\end{displaymath}
holds for every $x \geq 149$. We conclude for the case $149 \leq x \leq e^{390}$ by comparing the right hand side of the last inequality with  the right hand side of \eqref{2.5}. For the remaining case $3 \leq x \leq 149$, we check the desired inequality with a computer.
\end{proof}

Now, we use Proposition \ref{prop201} to obtain the following result concerning an explicit estimates for the distance between $x$ and $\vartheta(x)$, which we use in the proof of Theorem \ref{thm101}.

\begin{kor} \label{kor202}
For every $x \geq 2$, we have
\begin{displaymath}
\vert \vartheta(x) - x \vert < \frac{580115x}{\log^5 x}.
\end{displaymath}
\end{kor}

\begin{proof}
We use Proposition \ref{prop201} to get that the required inequality 
%\begin{displaymath}
%|\vartheta(x) - x| <  \frac{\sqrt{8}}{\sqrt{\pi \sqrt{R}}} \; g_1(e^{5801.149})\; \frac{x}{\log^5 x} < \frac{580115x}{\log^5 x}
%\end{displaymath}
holds for every $x \geq e^{5801.149}$.
%, where $g_1(x) = (\log x)^{21/4} e^{-\sqrt{(\log x)/R}}$ is a monotonic decreasing function for every $x \geq e^{441R/4}$.
In \cite[Proposition 2.5]{axler2017}, it is shown that the inequality $\vert \vartheta(x) - x \vert < 100 x/\log^4 x$ holds for every $x \geq 70\,111$, which implies the validity of the required inequality for every $70\,111 \leq x \leq e^{5801.15}$. For the remaining cases, we use a computer.
\end{proof}

\section{Proof of Theorem \ref{thm101}}

Let $k$ be a positive integer, $\eta_k$ and $x_1(k) \geq 2$ positive real numbers so that
\begin{equation}
|\vartheta(x) - x| < \frac{\eta_kx}{\log^k x} \tag{3.1} \label{3.1}
\end{equation}
for every $x \geq x_1(k)$ (The existence of such parameters is guaranteed by \eqref{2.4}). By \eqref{2.1}, we have
\begin{displaymath}
\pi(x) = \pi(x_1(k)) - \frac{\vartheta(x_1(k))}{\log x_1(k)} + \frac{\vartheta(x)}{\log x} +  \int_{x_1(k)}^{x}{\frac{\vartheta(t)}{t\log^{2} t}\ dt}.
\end{displaymath}
Now, we use \eqref{3.1} to derive
\begin{equation}
J_{k,-\eta_k,x_1(k)}(x) \leq \pi(x) \leq J_{k,\eta_k,x_1(k)}(x) \tag{3.2} \label{3.2}
\end{equation}
for every $x \geq x_1(k)$, where
\begin{align}
J_{k,\eta_k,x_1(k)}(x) & = \pi(x_1(k)) - \frac{\vartheta(x_1(k))}{\log x_1(k)} + \frac{x}{\log x} + \frac{\eta_k x}{\log^{k+1} x} + \int_{x_1(k)}^{x}{\left( \frac{1}{\log^{2} t} + \frac{\eta_k}{\log^{k+2} t} \ dt \right)}. \tag{3.3} \label{3.3}
\end{align}
The function $J_{k,\eta_k,x_1(k)}$ given in \eqref{3.3} was already introduced by Rosser and Schoenfeld \cite[p.81]{rosser1962} (for the case $k=1$) and Dusart \cite[p. 9]{dusart2010} and plays an important role in the following proof of Theorem \ref{thm101}

%\begin{thm} \label{thm301}
%The inequality
%\begin{equation}
%\pi(x) > \frac{x}{\log x - 1 - \frac{1}{\log x} - \frac{3}{\log^2 x}} \tag{3.4} \label{3.4}
%\end{equation}
%holds for every $65405887 \leq x \leq 2.73 \cdot 10^{40}$ and every $x \geq e^{580044/13}$.
%\end{thm}

\begin{proof}[Proof of Theorem \ref{thm101}]
First, we verify the validity of the required inequality, i.e.
\begin{equation}
\pi(x) > \frac{x}{\log x - 1 - \frac{1}{\log x} - \frac{3}{\log^2 x}}, \tag{3.4} \label{3.4}
\end{equation}
for every $x \geq e^{580044/13}$. For this, let $k=5$, $x_1 = 10^{13}$ and 
\begin{displaymath}
f(x) = \frac{x}{\log x - 1 - \frac{1}{\log x} - \frac{3}{\log^2x} - \frac{13}{\log^3 x} + \frac{580044}{\log^4 x}}.
\end{displaymath}
Further, we set $g(x) = J_{5, -580115, x_1}(x) - f(x)$. Then,
\begin{displaymath}
g'(x) = \frac{s(\log x)}{(\log^5x - \log^4x - \log^3x - 3\log^2x - 13\log x + 580044)^2\log^7x},
\end{displaymath}
where
\begin{align*}
s(y) & = 580\,576y^{10} - 6\,381\,045y^9 - 4\,060\,210y^8 - 15\,661\,259y^7 - 336\,607\,082\,789y^6 \\
& \phantom{\quad\quad} + 4\,037\,979\,215\,095y^5 - 2\,691\,881\,529\,325y^4 - 1\,345\,840\,694\,825y^3 \\
& \phantom{\quad\quad} - 1\,345\,478\,703\,065y^2 - 195\,224\,040\,181\,960\,440y + 975\,901\,480\,963\,513\,200.
\end{align*}
Since $s(y) > 0$ for every $y \geq \log x_1 \geq 28$, we get that 
\begin{equation}
J'_{5, -580115, x_1}(x) \geq f'(x) \tag{3.5} \label{3.5}
\end{equation}
for every $x \geq x_1$. By Dusart \cite[Table 6.1]{dusart2010}, we have $\vartheta(x_1) \leq 9\,999\,996\,988\,294$. Since $\pi(x_1) = 346\,065\,536\,839$, we use \eqref{3.3} to get $J_{5, -580115, x_1}(x_1) - f(x_1) > 3 \cdot 10^8$. Together with \eqref{3.5}, we obtain that $J_{5, -580115, x_1}(x) > f(x)$ for every $x \geq x_1$. Now, we use \eqref{3.2} and Corollary \ref{kor202} to get that the inequality $\pi(x) \geq f(x)$ holds for every $x \geq x_1$, which implies the validity of \eqref{3.4} for every $x \geq e^{580044/13}$.

In the second step, we show that the inequality \eqref{3.4} is fulfilled for every $10^{12} \leq x \leq 2.73 \cdot 10^{40}$. In \cite[Theorem 3.8]{axler2017}, it is shown that
\begin{displaymath}
\pi(t) > \frac{t}{\log t - 1 - \frac{1}{\log t} - \frac{2.85}{\log^2t} - \frac{13.15}{\log^3t} - \frac{70.7}{\log^4t} - \frac{458.7275}{\log^5t} - \frac{3428.7225}{\log^6t}}
\end{displaymath}
for every $t \geq 19\,033\,744\,403$. A comparsion of the last right hand side with the right hand side of \eqref{3.4} implies that the desired inequality \eqref{3.4} holds for every $19\,033\,744\,403 \leq x \leq 2.73 \cdot 10^{40}$.

To complete the proof, we check with a computer that $\pi(p_n) > s(p_{n+1})$ for every $\pi(65\,405\,887) \leq n \leq \pi(19\,033\,744\,403) + 1$.
\end{proof}

Using a result of Schoenfeld \cite[Corollary 1]{schoenfeld1976}, we obtain the following result.

\begin{prop} \label{prop301}
Under the assumption that the Riemann hypothesis is true, the inequality \eqref{3.4} holds for every $x \geq 65\,405\,887$.
\end{prop}

\begin{proof}
We denote the right hand side of \eqref{3.4} by $g(x)$ and set $h(x) = - \log^8x + 208\pi \sqrt{x}\log^2x + 96\pi \sqrt{x}\log x + 144\pi \sqrt{x}$. Then, $h(x) > 0$ for every $x \geq 233\,671\,227\,509$. Further, we define $f(x) = \text{li}(x) - \sqrt{x} \log x /(8\pi) - g(x)$. Then, $f'(x) \geq h(x)/(16\pi\sqrt{x}(\log^3 x - \log^2x - \log x - 3)^2 \log x) > 0$ for every $x \geq 233\,671\,227\,509$. In addition, we have $f(10^{12}) > 0$. So,
\begin{equation}
\text{li}(x) - \frac{\sqrt{x}}{8\pi} \, \log x > \frac{x}{\log x - 1 - \frac{1}{\log x} - \frac{3}{\log^2 x}} \tag{3.6} \label{3.6}
\end{equation}
for every $x \geq 10^{12}$. Under the assumption that the Riemann hypothesis is true, Schoenfeld \cite[Corollary 1]{schoenfeld1976} showed that the inequality $\pi(x) > \text{li}(x) - \sqrt{x} \log x /(8\pi)$ holds for every $x \geq 2\,657$. We conclude by applying \eqref{3.6} and Theorem \ref{thm101}.
\end{proof}

\section{Proof of Theorem \ref{thm102}}

In this section, we give a proof of Theorem \ref{thm102}. Let $n$ be a positive integer and $R = 5.573412$. Proposition \ref{prop201} implies that the inequality
\begin{equation}
|\vartheta(x) - x| < \frac{a_n(x)x}{\log^n x} \tag{4.1} \label{4.1}
\end{equation}
holds for every $x \geq 3$, where the function $a_n: [2, \infty) \to (0, \infty)$ is defined by
\begin{displaymath}
a_n(x) = \frac{\sqrt{8}}{\sqrt{\pi \sqrt{R}}} \, (\log x)^{n + 1/4} e^{-\sqrt{(\log x)/R}}.
\end{displaymath}
A straightforward calculation shows that the function $a_n(x)$ has a global minimum at $x_0 = e^{(4n+1)^2R/4}$.
%
%\begin{prop} \label{prop401}
%Let $x_0 = x_0(n) = e^{(4n+1)^2R/4}$. Then $a'(x_0) = 0$, $a'(x) < 0$ for every $2 \leq x < x_0$, and $a'(x) > 0$ for every $x > x_0$.
%\end{prop}
%
%\begin{proof}
%The derivative of $a_n(x)$ is given by
%\begin{displaymath}
%a_n'(x) = - \frac{\sqrt{2}}{2\sqrt{R}\sqrt{\pi\sqrt{R}}} \cdot \frac{e^{-\sqrt{(\log x)/R}}}{x (\log x)^{5/4 - n}} \cdot \left( 2 \log x - \sqrt{R}(4n+1) \sqrt{\log x} \right)
%\end{displaymath}
%and the claim follows.
%\end{proof}
%
For the proof of Theorem \ref{thm104}, we need the following inequality involving the function $a_n(x)$.

\begin{prop} \label{prop401}
For every $x \geq 851$, we have
\begin{displaymath}
\int_3^x \frac{a_n(t)}{\log^{n+2}t} \, dt \leq \frac{\sqrt{2}}{\sqrt{\pi \sqrt{R}}} \cdot \frac{x}{(\log x)^{3/4} e^{\sqrt{\log x/R}}}.
\end{displaymath}
\end{prop}

\begin{proof}
Let $x \geq 851$. From the definition of $a_n(t)$, we have
\begin{displaymath}
\int_3^x \frac{a_n(t)}{\log^{n+2}t} \, dt = \frac{\sqrt{8}}{\sqrt{\pi \sqrt{R}}} \int_3^x (\log t)^{-7/4} e^{-\sqrt{\log t/R}} \, dt.
\end{displaymath}
The substitution $t  = e^{Ry}$ gives
\begin{equation}
\int_3^x \frac{a_n(t)}{\log^{n+2}t} \, dt = \frac{\sqrt{8}}{R\sqrt{\pi}} \int_{\log 3/R}^{\log x/R} \frac{e^{Ry}}{y^{7/4} e^{\sqrt{y}}} \, dy. \tag{4.2} \label{4.2}
\end{equation}
For convenience, we write $b = \log 3/R$ and $c = \log x/R$, and define $f : [3,c] \to (0, \infty), y \mapsto e^{Ry}/(y^{7/4}e^{\sqrt{y}})$. It is easy to see that the function $f$ is convex on the interval $[b,c]$. Hence,
\begin{equation}
\int_b^c f(y) \, dy \leq \frac{c-b}{2} (f(b) + f(c)). \tag{4.3} \label{4.3}
\end{equation}
The function $g : [3,\infty) \to (0, \infty), y \mapsto y/(y^{11/4}e^{\sqrt{y/R}})$ is strictly increasing for every $x \geq 22.75$ and fulfilled $g(851) \geq g(3)$. Hence $g(y) \geq g(3)$ for every $y \geq 851$, which is equivalent to $bf(c) \geq cf(b)$. Applying this inequality to \eqref{4.3}, we get
\begin{displaymath}
\int_b^c f(y) \, dy \leq \frac{cf(c)}{2},
\end{displaymath}
since $bf(b) \geq 0$. Together with \eqref{4.2} and the definition of the function $f$, we conclude the proof.
%Further, we have $f'(y) = e^{Ry}(4Ry - 2\sqrt{y} - 7)/(4y^{11/4}e^{\sqrt{y}})$. Let $y_0 \approx 0.36$ be the unique solution of $f'(y) = 0$. Then, $f'(y) > 0$ for every $y > y_0$. Since $f(b) < f(\log 92/R)$, we obtain that $f(b) < f(\log x/R) = f(c)$ for every $y \geq 92$. Applying this to \eqref{2.4}, we get
%\begin{displaymath}
%\int_{\log 2/R}^{\log x/R} f(y) \, dy \leq cf(c).
%\end{displaymath}
\end{proof}

Now, we use the identity \eqref{2.1} and Proposition \ref{prop401} to obtain the following estimates for the prime counting function.

\begin{prop} \label{prop402}
Let $c = 3\sqrt{2}/\sqrt{\pi \sqrt{R}}$. For every $x \geq 2$, we have
\begin{equation}
\pi(x) > \emph{li}(x) - \frac{cx}{(\log x)^{3/4}e^{\sqrt{\log x/R}}} \tag{4.4} \label{4.4}
\end{equation}
and
\begin{equation}
\pi(x) < \emph{li}(x) + \frac{cx}{(\log x)^{3/4}e^{\sqrt{\log x/R}}} - \emph{li}(2) + \frac{2}{\log 2}. \tag{4.5} \label{4.5}
\end{equation}
\end{prop}

\begin{proof}
%The starting point for the proof of \eqref{4.4} is the identity \eqref{2.1}.
%; i.e.
%\begin{displaymath}
%\pi(y) = \frac{\vartheta(y)}{\log y} + \int_2^y \frac{\vartheta(t)}{t \log^2 t} \, dt,
%\end{displaymath}
%which holds for every $y \geq 2$.
First, let $x \geq 851$. Since $\vartheta(t)/(t \log^2 t) > 0$ for every $t \geq 2$, we use the identity \eqref{2.1} to get
\begin{displaymath}
\pi(x) > \frac{\vartheta(x)}{\log x} + \int_{3}^x \frac{\vartheta(t)}{t \log^2 t} \, dt.
\end{displaymath}
Applying \eqref{4.1}, we obtain that the inequality
\begin{displaymath}
\pi(x) > \frac{x}{\log x} - \frac{a_n(x)x}{\log^{n+1} x} + \int_3^x \frac{dt}{\log^2 t} - \int_3^x \frac{a_n(t)}{\log^{n+2} t} \, dt
\end{displaymath}
holds. Together with Proposition \ref{prop401} and the identity
\begin{displaymath}
\int_3^x \frac{dt}{\log^2 t} = \text{li}(x) - \frac{x}{\log x} - \text{li}(3) + \frac{3}{\log 3},
\end{displaymath}
we obtain the inequality
\begin{displaymath}
\pi(x) > \text{li}(x) - \frac{a_n(x)x}{\log^{n+1} x} - \text{li}(3) + \frac{3}{\log 3} - \frac{\sqrt{2}}{\sqrt{\pi \sqrt{R}}} \cdot \frac{x}{(\log x)^{3/4}e^{\sqrt{\log x/R}}},
\end{displaymath}
which implies \eqref{4.4} for every $x \geq 851$, since $3/\log 3 - \text{li}(3) > 0$. For smaller values of $x$, we check the inequality \eqref{4.4} with a computer.

%In order to prove that the inequality \eqref{4.5} holds for every $x \geq 2$, we use \eqref{2.1} and the method of 
The identity \eqref{2.1} gives that the identity
\begin{equation}
\pi(y) - \text{li}(y) = \frac{\vartheta(y) - y}{\log y} + \frac{2}{\log 2} - \text{li}(2) + \int_2^y \frac{\vartheta(t) - t}{t \log^2 t} \, dt \tag{4.6} \label{4.6}
\end{equation}
holds for every $y \geq 2$. First we consider the case $x \geq 851$. By Büthe \cite[Theorem 2]{buethe}, we have $\vartheta(t) < t$ for every $1 \leq t \leq 10^{19}$. Hence, by \eqref{4.6} and \eqref{4.1}, 
\begin{displaymath}
\pi(x) - \text{li}(x) < \frac{a_n(x)x}{\log^{n+1} x} + \frac{2}{\log 2} - \text{li}(2) + \int_3^x \frac{a_n(t)}{\log^{n+2} t} \, dt.
\end{displaymath}
Using Proposition \ref{prop201}, we get
\begin{displaymath}
\pi(x) - \text{li}(x) < \frac{a_n(x)x}{\log^{n+1} x} + \frac{2}{\log 2} - \text{li}(2) + \frac{\sqrt{2}}{\sqrt{\pi \sqrt{R}}} \cdot \frac{x}{(\log x)^{3/4} e^{\sqrt{\log x/R}}}.
\end{displaymath}
Substituting the definition of $a_n(x)$, we get that the inequality \eqref{4.5} holds for every $x \geq 851$. Again, we check the required inequality for smaller values of $x$ with a computer.
\end{proof}

The function $x \mapsto x/ \log^{n+2}x$ is strictly increasing for every $x > e^{n+2}$ and tends to infinity as $x \to \infty$. Therefore, there exists a positive integer $A_0(n) \geq 2$ so that
\begin{displaymath}
\frac{x}{\log^{n+2}x} \geq \frac{1}{(n+1)!}\sum_{k \leq n+1} \frac{2(k-1)!}{\log^k2}
\end{displaymath}
for every $x \geq A_0(n)$ and we get the following proposition.

\begin{prop} \label{prop403}
Let $c = 3\sqrt{2}/\sqrt{\pi \sqrt{R}}$. Then, for every $x \geq \max \{27, A_0(n) \}$, we have
\begin{equation}
\pi(x) > \sum_{k =1}^{n+1} \frac{(k-1)!x}{\log^kx} - \frac{cx}{(\log x)^{3/4}e^{\sqrt{\log x/R}}} \tag{4.7} \label{4.7}
% + \frac{x}{\log^{n+1} x} \left( n! - a_n(x) - \frac{\sqrt{8}}{\sqrt{\pi \sqrt{R}}} \right) 
\end{equation}
and for every $x \geq 4$, we have
\begin{equation}
\pi(x) < \sum_{k=1}^n \frac{(k-1)!x}{\log^kx} + \frac{n!\sqrt{x}}{\log^{n+1}2} + \frac{n!2^{n+1}x}{\log^{n+1}x} + \frac{cx}{(\log x)^{3/4}e^{\sqrt{\log x/R}}} + d, \tag{4.8} \label{4.8}
\end{equation}
where $d = - \emph{li}(2) + 2/\log 2$.
\end{prop}

\begin{proof}
We start with the proof of \eqref{4.7}. Let $x \geq \max \{ 27, A_0(n) \}$. We use \eqref{1.2} to get
%, we get that the inequality
%\begin{displaymath}
%\text{li}(x) \geq \int_2^x \frac{dt}{\log t}
%\end{displaymath}
%holds. Integration by parts gives
%\begin{displaymath}
%\text{li}(x) \geq \sum_{k=1}^{n+1} \frac{(k-1)!x}{\log^kx} + (n+1)! \int_2^x \frac{dt}{\log^{n+2}t} - \sum_{k=1}^{n+1} \frac{2(k-1)!}{\log^k2},
%\end{displaymath}
%which is equivalent to the inequality
\begin{displaymath}
\text{li}(x) \geq \sum_{k=1}^{n+1} \frac{(k-1)!x}{\log^kx} + (n+1)! \int_2^3 \frac{dt}{\log^{n+2}t} + (n+1)! \int_3^x \frac{dt}{\log^{n+2}t} - \sum_{k=1}^{n+1} \frac{2(k-1)!}{\log^k2}.
\end{displaymath}
Notice that the function $t \mapsto 1/\log^mt$ is strictly decreasing on the interval $[2,x]$ for every positive integer $m$. Hence
\begin{equation}
\text{li}(x) \geq \sum_{k=1}^{n+1} \frac{(k-1)!x}{\log^kx} + \frac{(n+1)!}{\log^{n+2}3} + \frac{(x-3) \cdot (n+1)!}{\log^{n+2}x} - \sum_{k=1}^{n+1} \frac{2(k-1)!}{\log^k2}. \tag{4.9} \label{4.9}
\end{equation}
We have $1/\log^{n+2} 3 \geq 3/\log^{n+2}t$ for every $t \geq 27$. Applying this to \eqref{4.9}, we get
\begin{displaymath}
\text{li}(x) \geq \sum_{k=1}^{n+1} \frac{(k-1)!x}{\log^kx} +\frac{(n+1)! x}{\log^{n+2}x} - \sum_{k=1}^{n+1} \frac{2(k-1)!}{\log^k2}.
\end{displaymath}
Since $x \geq A_0(n)$, wo obtain that the inequality
\begin{displaymath}
\text{li}(x) \geq \sum_{k=1}^{n+1} \frac{(k-1)!x}{\log^kx}
\end{displaymath}
holds. Now use \eqref{4.4} to complete the proof of \eqref{4.7}.

Next, we check the validity of \eqref{4.8}. Let $x \geq 4$. Again, we use \eqref{1.2} and integration by parts to get
\begin{displaymath}
\text{li}(x) \leq 1.05 + \sum_{k=1}^n \frac{(k-1)!x}{\log^kx} + n! \int_2^x \frac{dt}{\log^{n+1}t} - \sum_{k=1}^n \frac{2(k-1)!}{\log^k2}.
\end{displaymath}
Since $1.05 \leq 2/\log 2$, we get that the inequality
\begin{equation}
\text{li}(x) \leq \sum_{k=1}^n \frac{(k-1)!x}{\log^kx} + n! \int_2^x \frac{dt}{\log^{n+1}t} \tag{4.10} \label{4.10}
\end{equation}
holds. In the first part of the proof, we note that the function $t \mapsto 1/\log^{n+1}t$ is strictly decreasing on the interval $[2,x]$. Therefore
\begin{displaymath}
\int_2^x \frac{dt}{\log^{n+1}t} = \int_2^{\sqrt{x}} \frac{dt}{\log^{n+1}t} + \int_{\sqrt{x}}^x \frac{dt}{\log^{n+1}t} \leq \frac{\sqrt{x}}{\log^{n+1}2} + \frac{2^{n+1}x}{\log^{n+1}x}.
\end{displaymath}
Together with \eqref{4.10} and \eqref{4.5}, we obtain that the required inequality \eqref{4.8} holds.
\end{proof}

Now, we give the proof of Theorem \ref{thm102} in which Proposition \ref{prop403} plays an important role.

%\begin{kor} \label{kor405}
%The inequality
%\begin{equation}
%\pi(x) > \frac{x}{\log x} + \frac{x}{\log^2 x} + \frac{2x}{\log^3 x} + \frac{6x}{\log^4 x} \tag{4.11} \label{4.11}
%\end{equation}
%holds for every $x$ such that $10384261 \leq x \leq 5.5 \cdot 10^{25}$ and every $x \geq e^{6719}$.
%\end{kor}

\begin{proof}[Proof of Theorem \ref{thm102}]
In the first step, we verify that the inequality
\begin{equation}
\pi(x) > \frac{x}{\log x} + \frac{x}{\log^2 x} + \frac{2x}{\log^3 x} + \frac{6x}{\log^4 x} \tag{4.11} \label{4.11}
\end{equation}
holds for every $x \geq e^{6719}$. Let $n= 4$. It is easy to see that we can choose $A_0(4) = 132\,718\,993$.
%We have
%\begin{displaymath}
%\frac{1}{(n+1)!}\sum_{k=1}^{n+1} \frac{2(k-1)!}{\log^k2} \leq 3.1.
%\end{displaymath}
%Since the function $t \mapsto t/\log^6t - 3.1$ is strictly increasing for every $t > e^6$ and fulfills $f(132718993) \geq 0$, we can chose $A_0(4) = 132718993$.
Further, we set $A_1(4) = e^{6719}$. Then,
\begin{displaymath}
\frac{cx}{(\log x)^{3/4}e^{\sqrt{\log x/R}}} \leq \frac{n!x}{(\log x)^{n + 1}}
\end{displaymath}
for every $x \geq A_1(4)$, where $R = 5.573412$ and $c = 3\sqrt{2}/\sqrt{\pi \sqrt{R}}$. Now we apply the last inequality to \eqref{4.7} and get that the inequality \eqref{4.11} holds for every $x \geq e^{6719}$.

Next, we verify that the inequality \eqref{4.11} is valid for every $10\,384\,261 \leq x \leq 2.73 \cdot 10^{40}$. We denote the right hand side of the inequality \eqref{4.11} by $U(x)$. For $y > 0$ let $R(y) = U(y)\log y/y$ and $S(y) = (y^4 - y^3 - y^2 - 3y)/y^3$. We have $S(t) > 0$ for every $t > 2.14$ and $y^5R(y)S(y) = y^6 - T(y)$, where $T(y) = 11y^2 + 12y + 18$. Then, by Theorem \ref{thm101},
\begin{equation}
\pi(x) > \frac{x}{S(\log x)} > \frac{x}{S(\log x)} \left( 1 - \frac{T(\log x)}{\log^6 x} \right) = U(x), \tag{4.12} \label{4.12}
\end{equation}
which completes the proof for every $65\,405\,887 \leq x \leq 2.73 \cdot 10^{40}$. Finally, we use a computer to check that $\pi(p_n) > U(p_{n+1})$ for every positive integer $n$ such that $\pi(10\,384\,261) \leq n \leq \pi(65\,405\,887)$.
\end{proof}

Finally, we use Proposition \ref{prop301} to obtain the following result concerning \eqref{4.11}.

\begin{prop} \label{prop404}
Under the assumption that the Riemann hypothesis is true, the inequality \eqref{4.11} holds for every $x \geq 10\,384\,261$.
\end{prop}

\begin{proof}
We assume that the Riemann hypothesis is true. By \eqref{4.12} and Proposition \ref{prop301} we get that the inequality \eqref{4.11} is valid for every $x \geq 65\,405\,887$. Finally, it suffices to apply Theorem \ref{thm102}.
\end{proof}

\section{The proof of Theorem \ref{thm103}}

In the following proof of Theorem \ref{thm103}, we use a recent result of Büthe \cite[Theorem 2]{buethe} and an explicit estimate for the prime counting function $\pi(x)$ obtained in \cite[Korollar 1.24]{axler2013}.

\begin{proof}[Proof of Theorem \ref{thm103}]
First, we check that the inequality \eqref{1.9} holds for every real $x$ such that $1.62 \cdot 10^{12} \leq x \leq 10^{19}$. By Büthe \cite[Theorem 2]{buethe}, we have
\begin{equation}
\pi(t) < \tl(t) \tag{5.1} \label{5.1}
\end{equation}
for every $t$ such that $2 \leq t \leq 10^{19}$. Further, we use \cite[Theorem 2]{buethe} to get that $\pi(t) > \tl(t) - 2.1204\sqrt{t}/\log t$ for every $t$ such that $5.94 \cdot 10^{11} \leq t \leq 10^{19}$. Together with \eqref{5.1}, we obtain that
\begin{displaymath}
\pi \left( \frac{x}{e} \right) - \frac{\pi(x)^2\log x}{ex} > \text{Ram}(x),
\end{displaymath}
where
\begin{displaymath}
\text{Ram}(x) = \tl \left( \frac{x}{e} \right) - \frac{2.1204\sqrt{x/e}}{\log(x/e)} - \tl(x^2) \frac{\log x}{ex}.
\end{displaymath}
We show that $\text{Ram}(x)$ is positive. In order to prove this, we first show that the derivative of $\text{Ram}(t)$ is positive for every $1.06 \cdot 10^{12} \leq t \leq 10^{19}$. A straightforward calculation gives
\begin{equation}
\text{Ram}'(t) = \frac{(\tl(t)\log(t/e) - t)^2}{et^2\log(t/e)} - \frac{1.0602(\log t - 3)}{e\log^2(t/e)\sqrt{t/e}}. \tag{5.2} \label{5.2}
\end{equation}
From \eqref{5.1} and the lower bound for the prime counting function given in \cite[Korollar 1.24]{axler2013}, it follows that $\tl(t)\log(t/e) - t > t/(\log t \log(t/e))$ for every $t$ such that $468\,049 \leq t \leq 10^{19}$. Combined with \eqref{5.2}, we obtain that the inequality
\begin{displaymath}
\text{Ram}'(t) > \frac{1}{e\log^2t \log^3(t/e)} - \frac{1.0602(\log t - 3)}{e\log^2(t/e)\sqrt{t/e}}
\end{displaymath}
holds for every $t$ such that $468\,049 \leq t \leq 10^{19}$. Since $\sqrt{y} \geq 1.0602\sqrt{e}\log^4 y$ for every $y \geq 1.06 \cdot 10^{12}$, we conclude that the derivative of $\text{Ram}(t)$ is positive for every $1.06 \cdot 10^{12} \leq t \leq 10^{19}$. Together with $\text{Ram}(1.62 \cdot 10^{12}) > 85.86$, we get that $\text{Ram}(x)$ is positive, which implies that Ramanujan's prime counting inequality \eqref{1.9} holds unconditionally for every $1.62 \cdot 10^{12} \leq x \leq 10^{19}$. It remains to show that the inequality \eqref{1.9} holds for every $38\,358\,837\,683 \leq x \leq 1.62 \cdot 10^{12}$ as well. Dudek and Platt \cite[Theorem 1.3]{dudekplatt} showed by computation that $x = 38\,358\,837\,682$ is the largest integer counterexample below $10^{11}$ and that there are no more failures at integer values before $1.15 \cdot 10^{16}$. Since $t \mapsto t/\log t$ is a strictly increasing function for every $t > e$, we get that the inequality \eqref{1.9} holds for every $x$ such that $38\,358\,837\,683 \leq x \leq 1.62 \cdot 10^{12}$ as well and conclude the proof.
\end{proof}

\section{The proof of Theorem \ref{thm104}}

Now we use Proposition \ref{prop403} to prove our second main result concerning Ramanujan's prime counting inequality, which is stated in Theorem \ref{thm104} .
%concerning the smallest positive intger $H_0$ so that the inequality \eqref{1.5}; i.e.
%\begin{displaymath}
%\pi(x)^2 < \frac{ex}{\log x} \pi \left( \frac{x}{e} \right)
%\end{displaymath}
%holds for every $x \geq H_0$.

%\begin{thm} \label{thm401}
%The inequality \eqref{1.5} holds unconditionally for every $x \geq e^{9032}$; i.e.
%\begin{displaymath}
%H_0 \leq e^{9032}.
%\end{displaymath}
%\end{thm}

\begin{proof}[Proof of Theorem \ref{thm104}]
First, let $R = 5.573412$ and let $a$ be a positive real number. Since there is a positive integer $A_1(n,a) \geq 2$ so that
\begin{displaymath}
e^{\sqrt{\log x/R}} \geq \left( \frac{\log x}{a} \right)^{n + 1/4}
\end{displaymath}
for every $x \geq A_1(n,a)$, Proposition \ref{prop403} implies that
\begin{equation}
\pi(x) > \sum_{k=1}^n \frac{(k-1)!x}{\log^k x} + \frac{(n!-ca^{n+1/4})x}{\log^{n+1}x}, \tag{6.1} \label{6.1}
\end{equation}
for every $x \geq \max \{ 27, A_0(n), A_1(n,a) \}$, and
\begin{equation}
\pi(x) < \sum_{k=1}^n \frac{(k-1)!x}{\log^k x} + \frac{x}{\log^{n+1}x} \left(\frac{n! \log^{n+1}x}{\sqrt{x}\log^{n+1}2} + n!2^{n+1} + ca^{n+1/4} + \frac{d\log^{n+1}x}{x} \right) \tag{6.2} \label{6.2}
\end{equation}
for every $x \geq \max \{4, A_1(n,a) \}$, where $d = - \text{li}(2) + 2/\log 2$.

Now, let $n = 6$ and let $x_0 = e^{9031}$. It is easy to show that $A_0(6) = 1\,657\,493\,059\,174$ is a suitable choice for $A_0(6)$.
%In the first step of the proof, we determine a value for $A_0(6)$. Notice that \begin{displaymath}
%\frac{1}{(n+1)!}\sum_{k=1}^{n+1} \frac{2(k-1)!}{\log^k2} \leq 4.22.
%\end{displaymath}
%The function $t \mapsto t/\log^8t - 4.22$ is strictly increasing for every $t > e^8$ and positive at $t_0 = 1\,657\,493\,059\,174$. Hence we can choose
%$A_0(6) = 1\,657\,493\,059\,174$.
Further, we set $a = 14.4086$. Then the function
\begin{displaymath}
t \mapsto t - R\left( n + \frac{1}{4} \right)^2 \left( \log t + \log \left (\frac{1}{a} \right) \right)^2
\end{displaymath}
is positive for every $t \geq 9\,031$ and we can choose $A_1(6, 14.4086) = x_0$. Using \eqref{6.1} and \eqref{6.2}, we get
\begin{displaymath}
\sum_{k = 1}^6 \frac{(k-1)!x}{\log^kx} - \frac{27158494x}{\log^7 x} < \pi(x) < \sum_{k = 1}^6 \frac{(k-1)!x}{\log^kx} + \frac{27251374x}{\log^7x}
\end{displaymath}
%and
%\begin{equation}
%\pi(x) < \sum_{k = 1}^6 \frac{(k-1)!x}{\log^kx} + \frac{27251374x}{\log^7x} \tag{3.2} \label{3.2}
%\end{equation}
for every $x \geq x_0$. Using these inequalities we conclude that the inequality
%\begin{displaymath}
%\frac{ex}{\log x} \pi \left( \frac{x}{e} \right) - \pi(x)^2 > \frac{x}{\log x} \left(\sum_{k = 1}^6 \frac{(k-1)!x}{\log^k(x/e)} - \frac{27158494x}{\log^7(x/e)} \right) - \left( \sum_{k = 1}^6 \frac{(k-1)!x}{\log^kx} + \frac{27251374x}{\log^7x} \right)^2
%\end{displaymath}
%holds for every $x \geq ex_0$. By expanding the right-hand side of the last inequality, we get
 \begin{equation}
\frac{ex}{\log x} \pi \left( \frac{x}{e} \right) - \pi(x)^2 > \frac{x^2 f(\log x)}{\log^{14} x (\log x - 1)^7} \tag{6.3} \label{6.3}
\end{equation}
holds for every $x \geq ex_0$, where
\begin{align*}
f(y) & = y^{15} + 7y^{14} - 81\,660\,454y^{13} + 327\,013\,544y^{12} - 872\,039\,437y^{11} + 1\,199\,056\,017y^{10} \\
& \p{\q\q}  - 1\,308\,062\,388y^9 - 1\,199\,031\,244y^8 - 742\,610\,678\,698\,880y^7 + 5\,198\,360\,646\,460\,072y^6 \\
& \p{\q\q} - 15\,595\,195\,794\,997\,976y^5 + 25\,992\,104\,849\,073\,228y^4 - 25\,992\,179\,953\,690\,916y^3  \\
& \p{\q\q} + 15\,595\,340\,608\,417\,428y^2 - 5\,198\,455\,153\,885\,372y + 742\,637\,384\,887\,876.
\end{align*}
Now, it is easy to verify that $f(y) > 0$ for every $y \geq 9\,032$. Applying this to \eqref{6.3}, we get that Ramanujan's prime counting inequality \eqref{1.9} holds for every $x \geq ex_0 = e^{9032}$, as desired.
\end{proof}

\begin{rema}
Recently, Platt and Trudgian announced that they have fixed the error in the proof of \eqref{1.12} and even managed to improve the result in Theorem \ref{thm104} by showing
\begin{displaymath}
H_0 \leq e^{8801.037}.
\end{displaymath}
\end{rema}

\section{On a generalization of Ramanujan's prime counting inequality}

% Next, we use Theorem \ref{thm102} to establish a result concerning a generalization of the inequality \eqref{1.5}. For this purpose, l
Let $n$ be a positive integer and let $\Xi_n : (1, \infty) \to \R$ be given by
\begin{displaymath}
\Xi_n(x) = \prod_{k=1}^n \left( 1 - \frac{k-1}{\log x} \right)^{2^{n-k}}.
\end{displaymath}
In 2013, Hassani \cite[Theorem 1]{hassani2013} defined
\begin{displaymath}
R_n^{\Xi}(x) = \frac{e^n}{\Xi_n(x)} \left( \frac{x}{\log x} \right)^{2^n-1} \pi \left( \frac{x}{e^n} \right) - \pi(x)^{2^n}.
\end{displaymath}
and showed by induction that $R_n^{\Xi}(x) > 0$ for every $x \geq e^{n-1}x_R$, whenever Ramanujan's prime counting inequality \eqref{1.9} holds for every $x \geq x_R$ (For $n=1$, the inequality $R_1^{\Xi}(x) > 0$ is equivalent to the inequality \eqref{1.9}). Together with Theorem \ref{thm103}, Theorem \ref{thm104} and \eqref{1.11}, respectively, we obtain the following result.

\begin{prop} \label{prop701}
Let $n$ be a positive integer. Then the following hold:
\begin{enumerate}
\item[(i)] The inequality $R_n^{\Xi}(x) > 0$ holds for every $x$ such that $38\,358\,837\,683e^{n-1} \leq x \leq 10^{19}e^{n-2}$ and for every $x \geq e^{9031 + n}$.
\item[(ii)] Under the assumption that the Riemann hypothesis is true, we have $R_n^{\Xi}(x) > 0$ for every $x \geq 38\,358\,837\,683e^{n-1}$.
\end{enumerate}
\end{prop}

\begin{proof}
For (i), we follow the proof of Theorem 1 in \cite[p. 150]{hassani2013} and use Theorems \ref{thm103} and \ref{thm104}, respectively. Analogously, by using \eqref{1.11}, we conclude the proof of (ii).
\end{proof}

\subsection*{Acknowledgements}
I would like to thank David Platt for the fruitful conversations on this subject.
%I also wishes to thank the referee for helpful comments.
Furthermore, I would like to thank Mehdi Hassani for drawing my attention to the present subject.

\end{document}